\documentclass{article}  
\usepackage{amsmath} 
\usepackage{amsthm} 
\usepackage{tikz}
\usetikzlibrary{arrows}

\newtheorem{thm}{Theorem}
\newtheorem{claim}[thm]{Claim}
\newtheorem{lem}[thm]{Lemma}

\newenvironment{allintypewriter}{\ttfamily}{\par}

\title{Multicolor, multipartite Ramsey numbers for quadrilateral}
\date{}

\author{Janusz Dybizba\'nski\\
\small Institute of Informatics, Faculty of Mathematics, Physics and Informatics,\\[-0.8ex] \small University of Gda\'nsk, Wita  Stwosza 57, Gda\'nsk, Poland\\[-0.8ex]
\small \texttt{jdybiz@inf.ug.edu.pl}\\[2ex]
Yaser Rowshan\\
\small Institute for Advanced Studies in Basic Sciences (IASBS), \\[-0.8ex] \small Department of Mathematics, Zanjan 45137-66731, Iran\\[-0.8ex]
\small \texttt{y.rowshan@iasbs.ac.ir,~~~y.rowshan.math@gmail.com}\\
}

\begin{document}
\maketitle
\begin{abstract}
\noindent\noindent The $p$-partite Ramsey number for quadrilateral, denoted by $r_p(C_4,k)$, is the least positive integer $n$ such that any coloring of the edges of a complete $p$-partite graph with $n$ vertices in each partition with $k$ colors will result in a monochromatic copy of $C_4$. 
In this paper, we present an upper bound for $r_p(C_4,k)$ and the exact values of $r_p(C_4,2)$ for all $p\geq2$. In tripartite case we show that $r_3(C_4,k) \leq \lfloor (k+1)^2/2\rfloor-1$ and the exact value of 4-color tripartite Ramsey number $r_3(C_4,4)=11$.
\end{abstract}

\section{Introduction}

In this paper all graphs considered are undirected, finite and
contain neither loops nor multiple edges. Let $G$ be such a graph.
The vertex set of $G$ is denoted by $V(G)$, the edge set of $G$
by $E(G)$, and the number of edges in $G$ by $e(G)$. Let $\deg(v)$ be the degree of vertex $v$ and $\deg_i(v)$ denote the number of the edges incident to $v$ colored with color $i$. The open neighborhood in color $i$ of vertex $v$ is $N_i(v) = \{u \in V(G) | \{u,v\} \in E(G) \textrm { and \{u,v\} is colored with color $i$} \}$. For vertex $v\in V(G)$ and $U\subset V(G)$ let $\deg_i(v,U)$ be the number of edges in color $i$ between $v$ and vertices from $U$. Define $G[S]$ to be the subgraph of $G$ induced by a set of vertices $S \subset V(G)$ and $G^{i}$ to be the graph containing the edges of $G$ in color $i$. Let $C_n$ be the cycle on $n$ vertices, $K_n$ be the complete graph on $n$ vertices, and $K_n^p$ be the complete $p$-partite graph with every partition of size $n$.

For given graph $G$ and $k \geq 2$, the \emph{multicolor Ramsey number} $R_k(G)$
is the smallest integer $n$ such that if we arbitrarily color
with $k$ colors the edges of the complete graph $K_n$,
then it contains a monochromatic copy of $G$ in one of the colors. In this paper we consider Ramsey number for quadrilateral ($G=C_4$). In 1972 Chv\'atal and Harary~\cite{CH1} determined $R_2(C_4) = 6$. Bialostocki and Sch\"{o}nheim ~\cite{BiS} proved $R_3(C_4)=11$. Lower bound for $R_4(C_4)=18$ was determined by Exoo~\cite{Ex2} and upper bound by Sun Yongqi {\it et al.}~\cite{SunYLZ1}. Lazebnik and Woldar~\cite{LaWo} showed the best known bounds for $5$-color Ramsey number for quadrilateral $27 \leq R_5(C_4) \leq 29$. More bounds and general results for Ramsey number we can found in regularly updated survey by Radziszowski~\cite{Rad}.

The \emph{$p$-partite Ramsey number} for quadrilateral, denoted by $r_p(C_4,k)$, is the least positive integer $n$ such that any coloring of the edges of 
$K_n^p$
with $k$ colors will result in a monochromatic copy of $C_4$ in one of the colors. 

The study of bipartite Ramsey numbers ($p=2$) was initiated by Beineke and Schwenk in 1976, and continued by others, in particular Exoo \cite{Ex1}, Hattingh and Henning \cite{HaH}, Goddard, Henning, and Oellermann \cite{GHO}, and Lazebnik and Mubayi \cite{LaM}. The exact values of $r_2(C_4,k)$ are known for $k\leq 4$. Beineke and Schwenk~\cite{BeS} proved that $r_2(C_4,2)=5$, Exoo~\cite{Ex1} found the value $r_2(C_4,3)=11$ and, independently, Steinbach and Posthoff~\cite{SP} and Dybizba\'nski, Dzido and Radziszowski~\cite{DDR} showed that $r_2(C_4,4)=19$. We also know that for every integer $k \geq 5$, $k^2+1 \leq r_2(C_4,k) \leq k^2+k-2$~\cite{DDR}.

In tripartite case ($p=3$) exact value of Ramsey numbers $r_3(C_4,k)$ is known for $k\leq 3$. In 2014 Buada, Samana and Longani~\cite{BSL} prove that  $r_3(C_4,3)=7$. In the same paper one can find a proof that $r_3(C_4,2)=3$.

In this paper we show upper bound for $r_p(C_4,k)$ (Theorem~\ref{gora}), exact values of $r_p(C_4,2)$ for all $p\geq 2$ (Theorem~\ref{dwakoloryduzopartycji}). In tripartite case we show that $r_3(C_4,k) \leq \lfloor (k+1)^2/2\rfloor-1$ (Theorem~\ref{3partgora}) and determine the exact value of 4-color tripartite Ramsey number $r_3(C_4,4)=11$ (Theorem~\ref{twc411}).
\section{Upper bound}

The well known version of Cauchy-Schwarz theorem says that if $a_1,...,a_n$ is a sequence of non-negative real numbers and $M=\sum_{i=1}^n a_i$, then $\sum_{i=1}^n a_i^2 \geq M^2/n$. Moreover, $\sum_{i=1}^n a_i^2 = M^2/n$ if and only if $a_i = M/n$ for every $i$. We will use another (integer) version of this theorem. 
\begin{thm}
Let $a_1,...,a_n$ be a sequence of non-negative integers and $M=\sum_{i=1}^n a_i$, then 
$$\sum_{i=1}^n \binom{a_i}{2} \geq r\binom{a+1}{2} + (M-r) \binom{a}{2},$$ where
$a$ and $r$ are integers such that $M=an+r$ and $0\leq r < n$. Moreover, the minimum is reachable if and only if $|a_i-a_j|\leq 1$ for every $1\leq i < j \leq n$.
\label{CS}
\end{thm}
Note that: (1) $\sum_{i=1}^n \binom{a_i}{2}$ has the minimum value if and only if $\sum_{i=1}^n a_i^2$ has the minimum value. (2) If $a_j-a_i > 1$ for some $i\neq j$ then setting $a_j'=a_j-1$ and $a_i'=a_i+1$ we obtain the sequence with smaller $\sum_{i=1}^n a_i^2$.

\begin{thm}
Let $p,k\geq 2$ and $n\geq 1$ be integers, $w=\left\lceil \frac{2\Big\lceil \frac{n^2p(p-1)}{2k} \Big\rceil}{p} \right\rceil$, and $a$ and $r$ are integers such that $w=a(p-1)n+r$ and $0\leq r<(p-1)n$. If
\begin{equation}
\label{jeden}
r\binom{a+1}{2}+((p-1)n-r)\binom{a}{2}>\binom{n}{2}
\end{equation}
\label{gora}
then $r_p(C_4,k) \leq n$.
\end{thm}
\begin{proof}
Let $G=(V_1 \cup ... \cup V_p, E)$ be $p$-partite graph of order $pn$. For every $1\leq i \leq p$, $V_i$ is an independent set and $|V_i|=n$. By the pigeonhole principle, there exist $i$ ($1\leq i \leq p$), such that 
\begin{equation}
\label{dwa}
\sum_{v\in V_i} \deg(v) \geq \Big\lceil \frac{2|E|}{p} \Big\rceil
\end{equation}
Complete $p$-partite graph $K_n^p$ has $n^2p(p-1)/2$ edges and in every $k$-edge-coloring there exists a color, say $b$ (blue), containing at least 
$$\left\lceil \frac{n^2p(p-1)}{2k} \right\rceil$$ edges.\\

\noindent\noindent Moreover, by~\eqref{dwa}, there exists a partition $j$, say $j=1$, such that
$$\sum_{v\in V_1} \deg_b(v) \geq \left\lceil \frac{2\left\lceil \frac{n^2p(p-1)}{2k} \right\rceil}{p} \right\rceil = w$$
Let $U=V\setminus V_1 = \cup_{l=2}^{p}{V_l} = \{ u_1, u_2, ... ,u_{(p-1)n} \}$. For every $i$, $1\leq i \leq (p-1)n$, $N_b(u_i) \cap V_1$ is the vertex set of blue neighbors of $u_i$ in $V_1$. Denote by $a_i = |N_b(u_i) \cap V_1| = \deg_b(u_1,V_1)$. Vertex $u_i$ has $a_i$ blue neighbors in $V_1$ and $\binom{a_i}{2}$ pairs of blue neighbors in $V_1$. If $\sum_{i=1}^{(p-1)n} \binom{a_i}{2} > \binom{n}{2}$ then, by pigeonhole principle, there exists two vertices $u_i$ and $u_j$ with a common pair of blue neighbors, say $y_1, y_2$ in $V_1$, and there is blue cycle $(u_i,y_1,u_j,y_2)$ in the graph. Since $\sum_{i=1}^{(p-1)n}a_i = \sum_{u\in U} \deg_b(u, V_1) = \sum_{v\in V_1} \deg_b(v) \geq w$, by Theorem~\ref{CS}, 
$$\sum_{i=1}^{(p-1)n} \binom{a_i}{2} \geq r\binom{a+1}{2}+(a(p-1)n)\binom{a}{2}$$
$$\geq r\binom{a+1}{2}+((p-1)n-r)\binom{a}{2}.$$
Thus, if $r\binom{a+1}{2}+((p-1)n-r)\binom{a}{2}>\binom{n}{2}$ then every $k$-edge-coloring of $K_n^p$ contains a monochromatic cycle $C_4$.
\end{proof}

For every $p,k\geq 2$ inequality~\eqref{jeden} is satisfied if $n$ is big enough (for example $n=k^2$, for $p \geq 3$). For $p=3$ and $k=4$, $n=11$ is the smallest integer satisfy~\eqref{jeden}, so $r_3(C_4,4) \leq 11$. In a similar way we can bound other numbers. Table~\ref{tab1} presents examples of upper bounds for $k$-color, $p$-partite Ramsey numbers $r_p(C_4,k)$ which can be obtained from Theorem~\ref{gora} (for $2\leq p \leq 6$ and $2\leq k \leq 10$). In the bipartite case, we obtain results that are equal to the exact value for $k\leq 4$ and one worse than the best known bound ($k^2+k+1$~\cite{DDR}) for $k\geq 5$.

\begin{table}
\begin{tabular}{c|cccccc}
$k$ $\backslash$ $p$& bipartite &  tripartite &  $4$-partite &  $5$-partite &  $6$-partite \\
\hline
2 colors & \textbf{5} &          \textbf{3} &          3 &          3 &          3 \\
3	& \textbf{11} &         \textbf{7} &          5 &          \textbf{4} &          4 \\
4	& \textbf{19} &         \textbf{11} &         9 &          7 &          6 \\
5	& 29 &         17 &         12 &         11 &         9 \\
6	& 41 &         23 &         17 &         14 &         13 \\
7	& 55 &         31 &         23 &         18 &         16 \\
8	& 71 &         39 &         28 &         23 &         20 \\
9	& 89 &         49 &         35 &         29 &         24 \\
10	& 109 &        59 &         43 &         34 &         29 \\
\end{tabular}
\caption{Upper bounds for $r_p(C_4,k)$}
\label{tab1}
\end{table}

\section{Two colors multipartite Ramsey numbers}

\begin{lem} For positive integer $p$,
$$p \geq R_k(C_4) \text{ if and only if }r_p(C_4,k) = 1.$$
\label{duzopartycji}
\end{lem}
\begin{proof}
If $p \geq R_k(C_4)$ then every $k$-edge-coloring of $K_p = K_1^p$ contains a monochromatic $C_4$ so $r_p(C_4,k) = 1$. \\
If $1 \leq p < R_k(C_4)$, then there exist a $k$-edge-coloring of $K_1^p = K_p$ without monochromatic $C_4$ so $r_p(C_4,k) > 1$. 
\end{proof}

Since $K_n^{p-1}$ is a subgraph of $K_n^p$ and the latter is a subgraph of $K_{n+1}^p$, the following two monotonic properties hold:
\begin{itemize}
    \item If $k_1>k_2$ then $r_p(C_4,k_1) \geq r_p(C_4,k_2).$
    \item If $p_1>p_2$ then $r_{p_1}(C_4,k) \leq r_{p_2}(C_4,k).$
\end{itemize}



\begin{thm}{Two-color multipartite Ramsey number for quadrilateral}
$$r_p(C_4,2) = 
\begin{cases}
5 & \text{ for } p=2,\\
3 & \text{ for } p=3,\\
2 & \text{ for } p=4 \text{ or } 5,\\
1 & \text{ for } p \geq 6.
\end{cases}$$
\label{dwakoloryduzopartycji}
\end{thm}
\begin{proof}
Beineke and Schwenk in~\cite{BeS} show that $r_2(C_4,2)=5$.
Buada {\it et al.}~\cite{BSL} show that $r_3(C_4,2)=3$.
Since $R_2(C_4)=6$~\cite{CH1}, then by Lemma~\ref{duzopartycji}, we have $r_p(C_4,2) = 1$, for $p \geq 6$, and (by Lemma~\ref{duzopartycji} and monotonic) $r_4(C_4,2) \geq r_5(C_4,2) > 1$.

To finish the proof, we shall show that every $2$-edge-coloring of $K_2^4$ with partition sets $X_1=\{x_1,x_2\},X_2=\{w_1,w_2\},X_3=\{y_1,y_2\}$, and $X_4=\{z_1,z_2\}$, contains monochromatic $C_4$. Consider any $2$- edges coloring of $G = K_2^4=(\{x_1,x_2,w_1,w_2,y_1,y_2,z_1,z_2\},E(G))$ say $(G^r=(V(G), E_r), G^b=(V(G), E_b))$, where
\[E_r\cup E_b=E(G)=\binom{V(G)}{2} \setminus \{\{x_1,x_2\}, \{w_1,w_2\},  \{y_1,y_2\}, \{z_1,z_2\}\}.\]
Now, we have the following claims:
\begin{claim}
If there exists a vertex of $V(G)$, say $x_1\in X_1$, and a color, say red, such that $\deg_r(x)\geq 4$ then the coloring contains monochromatic $C_4$.
\label{claim1}
\end{claim}
\begin{proof}
Considering vertex $x_2$. As we aim to avoid $C_4$ then at least 3 edges from $x_2$ to $N_r(x_1)$ must be colored blue (denoted as color $b$). Consider three vertices 
$\{v_1,v_2,$ $v_3\} \subset N_r(x_1) \cap N_b(x_2)$. \\
a) If those vertices are in different partition then consider the triangle $\{v_1,v_2,v_3\}$. Two of the edges of this triangle must have the same color. 
These two edges create monochromatic $C_4$ with $x_1$ (if they are red) or $x_2$ (if blue).\\
b) Suppose that two of those vertices are in the same partition. Without loss of generality we can assume that $\{v_1,v_2,v_3\} = \{y_1,y_2,z_1\}$. Consider vertex $w_1$ and colors of the edges $\{w_1,y_1\}$, $\{w_1,y_2\}$ and $\{w_1,z_1\}$. Two of these edges must have the same color,  creating a monochromatic $C_4$ with $x_1$ or $x_2$.
\end{proof}
In the sequel we assume that every vertex in each color has degree 3. 
\begin{claim}
If there exists a vertex of $V(G)$, say $x_1\in X_1$, and a color say $b$, such that $X_i\subseteq N_b(x_1)$ for one $i\in\{2,3,4\}$, then the coloring contains monochromatic $C_4$.
\label{claim2}
\end{claim}
\begin{proof}
Without loss of generality let $N_b(x_1)= \{w_1,w_2, y_1\}$, therefore $N_r(x_1)= \{y_2,z_1,z_2\}$. Now considering $x_2$, one can check that $ \{w_1,w_2\}\subseteq N_r(x_2)$ and $ \{z_1,z_2\}\subseteq N_b(x_2)$. Otherwise, by contrary assume that $x_2z_1\in E_r$ (for other case the proof is same), hence as $\deg_b(z_1)=\deg_r(z_1)=3$, and $x_iz_1\in E_r$, then one can say that  $x_1$, $z_1$ and two vertices from $N_b(x_1) \cap N_b(z_1)$ create blue $C_4$.  If $x_2y_1\in E_b$ then as $\deg_r(y_1)=\deg_b(y_1)=3$, and $x_iy_1\in E_b$ one can say that either $x_1$, $y_1$ and two vertices from $N_r(x_1) \cap N_r(y_1)$ create red $C_4$ or $x_2$, $y_1$ and two vertices from $N_r(x_2) \cap N_r(y_1)$ create red $C_4$. So assume that $ N_r(x_1)=N_b(x_2)$ and $N_r(x_2)= N_b(x_1)$. For one $i=1,2$, $y_2z_i\in E_b$, otherwise $C_4\subseteq G^r$. Without loss of generality let  $y_2z_1\in E_b$. Therefore as $\deg_r(z_1)=3$, and $x_2z_1, y_2z_1\in E_b$ one can say that  $x_2$, $z_1$ and two vertices from $N_r(x_2) \cap N_r(z_1)$ create red $C_4$. Hence claim holds.
\end{proof}
Now by Claim~\ref{claim2}, without loss of generality we may suppose that  $N_b(x_1)=\{w_1,y_1,z_1\}$. Assume that $|N_b(x_1)\cap N_b(x_2)| \neq 0$, and without loss of generality let $w_1\in N_b(x_2)$. Hence by Claim~\ref{claim2}, $x_1w_2,x_2w_2\in E_r$.   Now, as $\deg_r(w_1)=3$ and $w_1x_i\in E_b$,  either $x_1$, $w_1$ and two vertices from $N_r(x_1) \cap N_r(w_1)$ create red $C_4$ or $x_2$, $w_1$ and two vertices from $N_r(x_2) \cap N_r(w_1)$ create red $C_4$.  So, let $N_b(x_2)=N_r(x_1)$ and $N_r(x_2)=N_b(x_1)$. Consider $w_1$ and without loss of generality assume that $w_1z_1\in E_b$. Therefore by Claim~\ref{claim2} we have $w_1z_2, w_2z_1\in E_r$ and $w_2z_2\in E_b$. If either $w_1y_1\in E_b$ or $y_1z_1\in E_b$ then one can check that $C_4\subseteq G^b[\{x_1, w_1,y_1,z_1\}]$. So we may suppose that $w_1y_1$, $y_1z_1\in E_r$, in this case we have $C_4\subseteq G^r[\{x_2, w_1,y_1,z_1\}]$, hence the proof is complete.
\end{proof}
\section{Tripartite Ramsey numbers}

\begin{thm} For every integer $k\geq 2$,
$$r_3(C_4,k) \leq 
\begin{cases} 
\frac{(k+1)^2-1}{2}-1 & \text{for even }k, \\
\frac{(k+1)^2}{2}-1 & \text{for odd }k. 
\end{cases}$$
\label{3partgora}
\end{thm}
\begin{proof}
For $k=2$ we know (from~\cite{BSL}) that $r_3(C_4,2) = 3 = \frac{(2+1)^2-1}{2}-1$. For even $k>2$, we use Theorem~\ref{gora}. Note that, for
$n=\frac{(k+1)^2-1}{2}-1$
\begin{itemize}
\setlength\itemsep{-1em}
\item $\Big\lceil \frac{n^2p(p-1)}{2k} \Big\rceil = \frac{3}{4}k^3+3k^2-5$\\
\item $w=\frac{1}{2}k^3+2k^2-3$,\\
\item $a=\frac{k}{2}$,\\
\item $r=k^2+k-3$,\\
\end{itemize}
and inequality $r\binom{a+1}{2}+((p-1)n-r)\binom{a}{2}>\binom{n}{2}$ can be simplified to $k>2$, so $r_p(C_4,k) \leq n$ for even $k\geq 2$.

For odd $k>1$ note that, for
$n=\frac{(k+1)^2}{2}-1$
\begin{itemize}
\setlength\itemsep{-1em}
\item $\Big\lceil \frac{n^2p(p-1)}{2k} \Big\rceil = \frac{3}{4}k^3 + 3k^2 + \frac{3}{2}k - \frac{9}{4} - \frac{1}{4} ((k-1)\mod 4)$\\
\item $w=\frac{k^3}{2}+ 2 k^2+ k -\frac{3}{2}  $,\\
\item $a=\frac{k+1}{2}$,\\
\item $r=\frac{k^2+k-2}{2}$,\\
\end{itemize}
and inequality $r\binom{a+1}{2}+((p-1)n-r)\binom{a}{2}>\binom{n}{2}$ can be simplified to $k(k+2)>3$ so $r_p(C_4,k) \leq n$ for odd $k\geq 3$.
\end{proof}

From~\cite{BSL} we know exact values $r_3(C_4,2)=3$ and $r_3(C_4,3)=7$. The next Theorem shows that the upper bound presented in Theorem~\ref{3partgora} is sharp for $k=4$.

\begin{thm}
$r_3(C_4,4)=11$.
\label{twc411}
\end{thm}
\begin{proof} From Theorem~\ref{3partgora} we have $r_3(C_4,4)\leq 11$. For the lower bound we present adjacency matrix of $4$-coloring of $K_{10}^3$ (Fig.~\ref{c411}). 
The entry in the $i$-th row and $j$-th column of the matrix refer to the color of edge $(i,j)$. Value $0$ means that there is no edge and values greater than $0$ are colors.

The matrix in Figure~\ref{c411} is written in partitioned form as\\
$$\begin{bmatrix}
0 & 0 & M_{1,2} & M_{1,3} & M_{1,4} & M_{1,5} \\
0 & 0 & M_{2,2} & M_{2,3} & M_{2,4} & M_{2,5} \\
M_{1,2}^T & M_{2,2}^T & 0 & 0 & M_{3,4} & M_{3,5} \\
M_{1,3}^T & M_{2,3}^T & 0 & 0 & M_{4,4} & M_{4,5} \\
M_{1,4}^T & M_{2,4}^T & M_{3,4}^T & M_{4,4}^T & 0 & 0 \\
M_{1,5}^T & M_{2,5}^T & M_{3,5}^T & M_{4,5}^T & 0 & 0 \\
\end{bmatrix},
$$\\
where each block $M_{i,j}$ in this partitioned matrix is $5\times 5$ matrix of the form:
$$\begin{bmatrix}
a & b & c & d & e \\
e & a & b & c & d \\
d & e & a & b & c \\
c & d & e & a & b \\
b & c & d & e & a \\
\end{bmatrix},
$$
\end{proof}

\begin{figure}[ht!]
\begin{center}
{\small
\begin{allintypewriter}
\begin{tabular}{c|c||c|c||c|c}
X 0 0 0 0&	0 0 0 0 0&	1 4 2 3 2&	2 1 3 4 4&	1 3 2 2 3&	4 1 1 4 3\\
0 X 0 0 0& 	0 0 0 0 0&	2 1 4 2 3&	4 2 1 3 4&	3 1 3 2 2&	3 4 1 1 4\\
0 0 X 0 0& 	0 0 0 0 0&	3 2 1 4 2&	4 4 2 1 3&	2 3 1 3 2&	4 3 4 1 1\\
0 0 0 X 0& 	0 0 0 0 0&	2 3 2 1 4&	3 4 4 2 1&	2 2 3 1 3&	1 4 3 4 1\\
0 0 0 0 X& 	0 0 0 0 0&	4 2 3 2 1&	1 3 4 4 2&	3 2 2 3 1&	1 1 4 3 4\\
\hline
0 0 0 0 0& 	X 0 0 0 0&	3 2 3 1 4&	3 1 1 2 4&	4 4 1 2 1&	3 2 4 2 3\\
0 0 0 0 0& 	0 X 0 0 0&	4 3 2 3 1&	4 3 1 1 2&	1 4 4 1 2&	3 3 2 4 2\\
0 0 0 0 0& 	0 0 X 0 0&	1 4 3 2 3&	2 4 3 1 1&	2 1 4 4 1&	2 3 3 2 4\\
0 0 0 0 0& 	0 0 0 X 0&	3 1 4 3 2&	1 2 4 3 1&	1 2 1 4 4&	4 2 3 3 2\\
0 0 0 0 0& 	0 0 0 0 X&	2 3 1 4 3&	1 1 2 4 3&	4 1 2 1 4&	2 4 2 3 3\\
\hline
\hline
1 2 3 2 4& 	3 4 1 3 2&	X 0 0 0 0&	0 0 0 0 0&	3 3 4 1 1&	4 4 1 2 2\\
4 1 2 3 2& 	2 3 4 1 3&	0 X 0 0 0&	0 0 0 0 0&	1 3 3 4 1&	2 4 4 1 2\\
2 4 1 2 3& 	3 2 3 4 1&	0 0 X 0 0&	0 0 0 0 0&	1 1 3 3 4&	2 2 4 4 1\\
3 2 4 1 2& 	1 3 2 3 4&	0 0 0 X 0&	0 0 0 0 0&	4 1 1 3 3&	1 2 2 4 4\\
2 3 2 4 1& 	4 1 3 2 3&	0 0 0 0 X&	0 0 0 0 0&	3 4 1 1 3&	4 1 2 2 4\\
\hline
2 4 4 3 1& 	3 4 2 1 1&	0 0 0 0 0&	X 0 0 0 0&	2 4 2 4 3&	1 3 1 3 2\\
1 2 4 4 3& 	1 3 4 2 1&	0 0 0 0 0&	0 X 0 0 0&	3 2 4 2 4&	2 1 3 1 3\\
3 1 2 4 4& 	1 1 3 4 2&	0 0 0 0 0&	0 0 X 0 0&	4 3 2 4 2&	3 2 1 3 1\\
4 3 1 2 4& 	2 1 1 3 4&	0 0 0 0 0&	0 0 0 X 0&	2 4 3 2 4&	1 3 2 1 3\\
4 4 3 1 2& 	4 2 1 1 3&	0 0 0 0 0&	0 0 0 0 X&	4 2 4 3 2&	3 1 3 2 1\\
\hline
\hline
1 3 2 2 3& 	4 1 2 1 4&	3 1 1 4 3&	2 3 4 2 4&	X 0 0 0 0&	0 0 0 0 0\\
3 1 3 2 2& 	4 4 1 2 1&	3 3 1 1 4&	4 2 3 4 2&	0 X 0 0 0&	0 0 0 0 0\\
2 3 1 3 2& 	1 4 4 1 2&	4 3 3 1 1&	2 4 2 3 4&	0 0 X 0 0&	0 0 0 0 0\\
2 2 3 1 3& 	2 1 4 4 1&	1 4 3 3 1&	4 2 4 2 3&	0 0 0 X 0&	0 0 0 0 0\\
3 2 2 3 1& 	1 2 1 4 4&	1 1 4 3 3&	3 4 2 4 2&	0 0 0 0 X&	0 0 0 0 0\\
\hline
4 3 4 1 1& 	3 3 2 4 2&	4 2 2 1 4&	1 2 3 1 3&	0 0 0 0 0&	X 0 0 0 0\\
1 4 3 4 1& 	2 3 3 2 4&	4 4 2 2 1&	3 1 2 3 1&	0 0 0 0 0&	0 X 0 0 0\\
1 1 4 3 4& 	4 2 3 3 2&	1 4 4 2 2&	1 3 1 2 3&	0 0 0 0 0&	0 0 X 0 0\\
4 1 1 4 3& 	2 4 2 3 3&	2 1 4 4 2&	3 1 3 1 2&	0 0 0 0 0&	0 0 0 X 0\\
3 4 1 1 4& 	3 2 4 2 3&	2 2 1 4 4&	2 3 1 3 1&	0 0 0 0 0&	0 0 0 0 X\\
\end{tabular}
\end{allintypewriter}
}
\end{center}
\caption{Matrix of $4$-edge-coloring of $K_{10}^3$ without monochromatic $C_4$.}
\label{c411}
\end{figure}

\section*{Acknowledgements}
I would like to thank Andrzej Szepietowski for comments on a preliminary version of this paper.


\newpage
\begin{thebibliography}{99}
\bibitem{BeS}
L. W. Beineke and A. J. Schwenk,
On a Bipartite Form of the Ramsey Problem,
{\it Proceedings of the Fifth British Combinatorial
Conference 1975, Congressus Numerantium},
{\bf XV} (1976) 17--22. 

\bibitem{BiS}
A. Bialostocki and J. Sch\"{o}nheim, 
On Some Tur\'an and Ramsey Numbers for C4, 
{\it Graph Theory and Combinatorics} (ed. B. Bollob\'as), 
Academic Press, London, (1984) 29--33.

\bibitem{BSL}
S. Buada, D. Samana, and V. Longani, 
The Tripartite Ramsey Numbers $r_t(C_4; 2)$ and $r_t(C_4; 3)$, 
{\it Italian Journal Of Pure And Applied Mathematics}
{\bf 33} (2014) 383--400

\bibitem{CH1}
V. Chv\'atal and F. Harary,
Generalized Ramsey Theory for Graphs, II. Small Diagonal Numbers,
{\it Proceedings of the American Mathematical Society},
{\bf 32} (1972) 389--394.

\bibitem{DDR} J. Dybizba\'nski, T. Dzido and S. Radziszowski, 
On Some Zarankiewicz Numbers and Bipartite Ramsey Numbers for Quadrilateral,
{\it Ars Combinatoria}, 
{\bf CXIX} (2015) 275--287.

\bibitem{Ex1}
G. Exoo,
A Bipartite Ramsey Number,
{\it Graphs and Combinatorics},
{\bf 7} (1991) 395--396.

\bibitem{Ex2}
G. Exoo,
Constructing Ramsey Graphs with a Computer,
{\it Congressus Numerantium},
{\bf 59} (1987) 31--36.

\bibitem{GHO}
W. Goddard, M. A. Henning, and O. R. Oellermann,
Bipartite Ramsey Numbers and Zarankiewicz Numbers,
{\it Discrete Mathematics}, 
{\bf 219} (2000) 85--95.

\bibitem{HaH}
J. H. Hattingh and M. A. Henning,
Bipartite Ramsey Theory,
{\it Utilitas Mathematica},
{\bf 53} (1998) 217--230.

\bibitem{LaM}
F. Lazebnik and D. Mubayi,
New Lower Bounds for Ramsey Numbers of Graphs and Hypergraphs,
{\it Advances in Applied Mathematics},
{\bf 28} (2002) 544--559.

\bibitem{LaWo}
F. Lazebnik and A. Woldar, 
New Lower Bounds on the Multicolor Ramsey Numbers $r_k(C_4)$, 
{\it Journal of Combinatorial Theory}, 
Series B, {\bf 79} (2000) 172--176.

\bibitem{Rad} S. P. Radziszowski, Small Ramsey Numbers,
{\it Electronic Journal of Combinatorics}, Dynamic Survey {\bf DS1},
revision \#16 (2021),
{\tt http://www.combinatorics.org}.

\bibitem{SunYLZ1}
Sun Yongqi, Yang Yuansheng, Lin Xiaohui, Zheng Wenping,
The Value of the Ramsey Number $R_4(C_4)$,
{\it Utilitas Mathematica},
{\bf 73} (2007) 33--44.

\bibitem{SP}
B. Steinbach and Ch. Posthoff,
Extremely Complex 4-Colored Rectangle-Free Grids: Solution of Open Multiple-Valued Problems, 
{\it Proceedings of the IEEE 42nd International Symposium on Multiple-Valued Logic}, Victoria, British Columbia, Canada, (2012) 37--44.

\end{thebibliography}
\end{document}